\documentclass[11pt]{amsart}
\usepackage{amsfonts}
\usepackage{amsmath}
\usepackage{amssymb}
\usepackage{multicol}
\usepackage{stmaryrd}
\usepackage{cite}
\usepackage{epsfig}
\usepackage{color}
\usepackage{graphics}
\usepackage{graphicx}
\usepackage{epstopdf}
\usepackage{multicol,graphics}
\newcommand\bes{\begin{eqnarray}}
\newcommand\ees{\end{eqnarray}}

\newtheorem{theorem}{Theorem}[section]
\newtheorem{lemma}[theorem]{Lemma}

\newtheorem{definition}[theorem]{Definition}
\newtheorem{remark}[theorem]{Remark}
\newtheorem{proposition}[theorem]{Proposition}
\numberwithin{equation}{section}
\allowdisplaybreaks

\begin{document}
\title[The Uniform Spreading Speed for Non-uniform Initial Data]{
\textbf{The Uniform Spreading Speed in Cooperative Systems  with   Non-uniform Initial Data}}

\author[Ru Hou, Zhian Wang, Wen-Bing Xu, Zhitao Zhang]{Ru Hou$^{1}$, Zhian Wang$^{2}$, Wen-Bing Xu$^{3,4,*}$, Zhitao Zhang$^{4,5}$}
\thanks{\hspace{-.53cm}
$^1$ School of Mathematical Sciences, Peking University, Beijing 100871, P. R. China.
\\
$^2$ Department of Applied Mathematics, The Hong Kong Polytechnic University, Hung Hom, Kowloon, Hong Kong.
\\
$^3$ School of Mathematical Sciences, Capital Normal University, Beijing 100048, PR China.\\
$^4$ Academy of Mathematics and Systems Science, Chinese Academy of Sciences, Beijing 100190, P. R. China.
\\
$^5$ School of Mathematical Sciences, University of Chinese Academy of Sciences, Beijing 100049, P. R. China.
\\
$^*${\sf Corresponding author} (6919@cnu.edu.cn)}

\date{\today}
\begin{abstract}
This paper considers  the spreading speed of cooperative nonlocal dispersal system with irreducible reaction functions and non-uniform initial data. Here the non-uniformity means that all components of initial data decay exponentially but their decay  rates are different.
It is well-known that in a monostable reaction-diffusion or nonlocal dispersal equation, different decay  rates of initial data yield  different  spreading speeds. In this paper, we show that due to the cooperation and irreducibility of reaction functions, all components of the solution with non-uniform initial data will possess a uniform spreading speed which decreasingly depends only on the smallest decay  rate of initial data.
The decreasing  property of the uniform spreading speed in the smallest decay  rate further implies that the component with  the smallest decay rate  can accelerate the spatial propagation of other components. In addition, all the methods in this paper can be carried over to the cooperative system  with classical diffusion (i.e. random diffusion).

\textbf{Keywords}: Nonlocal dispersals, spreading speeds, cooperative systems, exponential decay.

\textbf{Mathematics Subject Classification numbers}: 35C07, 35K57, 92D25.
\end{abstract}
\maketitle

\section{Introduction}
The long-range dispersal,  such as the spread of infectious disease across countries and continents by the travel of infected humans  \cite{HF2014}, has increasingly become an important phenomenon nowadays,
and it  has attracted extensive attention of researchers (see \cite{CMM2006,MMC2011,Mur1993}).
Mathematically the long-range dispersal can be  modelled by a nonlocal dispersal operator that describes  the  movements   between not only  adjacent   but also nonadjacent spatial locations. A typical nonlocal dispersal equation  with reaction is given by
\begin{equation}\label{1.3}
u_t=k*u-u+f(u),\quad t>0,~x\in\mathbb R,
\end{equation}
where $u(t,x)$ stands for the population density at location $x$ and time $t$, $f(u)$ is a reaction function, and the nonlocal dispersal operator is represented by
\[
k*u(t,x)-u(t,x)=\int_{\mathbb R}k(x-y)u(t,y)dy-u(t,x).
\]
Here $k: \mathbb{R} \to \mathbb{R}$ is a nonnegative and continuous function with $\int_{\mathbb{R}}k(x)dx=1$.
As stated in \cite{Fife2003},  $k(x-y)$ can be viewed as  the  probability for   individuals to move from  location $y$ to location $x$,  $k*u(t,x)=\int_{\mathbb R}k(x-y)u(t,y)dy$ stands for the rate at which individuals arrive at location $x$ from other locations,
and $-u(t,x)=-\int_{\mathbb R}k(y-x)u(t,x)dy$ is the rate at which individuals  leave location $x$ and move to other locations.
One of the most significant  research topics in the literature for \eqref{1.3} is the wave propagation phenomena which are associated with the studies of traveling wave solutions, entire solutions and spreading speeds. These results can be used to describe the spreading process of populations, such as the spatial spread of infectious diseases and the invasion of species.  For the traveling wave solutions   of \eqref{1.3}, we refer to the classical works by Bates et al. \cite{BFRW},  Carr and Chmaj \cite{CC2004}, Chen \cite{Chen1997}, Chen and Guo \cite{CG2003},  Coville, D\'{a}vila and Mart\'{\i}nez \cite{CDM2008},   Schumacher \cite{Sch1980}, Yagisita \cite{Yag2009}, etc. For the entire solutions of \eqref{1.3}, we refer to, for example, Li, Sun and Wang \cite{LSW2010}.
For the spreading speeds of \eqref{1.3}, we refer to the works by Lutscher, Pachepsky and Lewis \cite{LPL2005},  Shen and Zhang \cite{SZ2010}, Zhang, Li and Wang \cite{ZLW2012}, Rawal, Shen and Zhang\cite{RSZ2015}, Finkelshtein, Kondratiev and  Tkachov \cite{FKT2015,FKT2018},   Liang and Zhou \cite{LZ2020}, etc.

In this paper, we are concerned with  the spreading speed  of the following $m$-component nonlocal dispersal  system
\begin{equation}\label{1.1}
\left\{
\begin{aligned}
&U_t=D(K*U-U)+F(U),&&t>0,~x\in\mathbb R,\\
&U(0,x)=U_0(x)=(u_{1,0}(x),\ldots u_{m,0}(x)),&&x\in\mathbb R,
\end{aligned}
\right.
\end{equation}
where  $U=(u_1,\ldots,u_m)$,  $K=(k_1,\ldots,k_m)$,  $F=(f_1,\ldots,f_m)$, $D=\text{diag}\{d_1,\ldots,d_m\}$ with $d_j>0$, and $2\leqslant m\in \mathbb Z^+$.
The nonlocal dispersal is represented by
\[
K*U(t,x)-U(t,x)\triangleq(k_1*u_1(t,x)-u_1(t,x),\ldots,k_m*u_m(t,x)-u_m(t,x)).
\]
We assume that $F(U)$ is cooperative (namely  $\frac{\partial}{\partial u_i} f_j(U)\geqslant0 $ for any  $j\neq i$) and  monostable with an unstable equilibrium  $U\equiv\mathbf 0\in\mathbb R^m$ and a stable equilibrium   $U\equiv P\in(\mathbb R^+)^m$.
Assume that
\[
U_0(\cdot)\not\equiv\mathbf 0,~\mathbf 0\leqslant U_0(x)\leqslant P~~\text{for all}~x\in\mathbb R.
\]
The kernel $K\in C(\mathbb R,\mathbb R^m)$ is symmetric on $\mathbb R$ and satisfies   the Mollison condition (see \cite{CDM2008,Mur1993,Mollison}), in the sense that, there exists  $\Lambda>0$  such that
\[
\int_\mathbb{R}k_j(x)e^{\Lambda|x|}dx<+\infty,~j\in\{1,\ldots,m\}.
\]
The local dispersal system, as a counterpart of \eqref{1.1}, is called the reaction-diffusion system which reads as
\begin{equation}\label{1.5}
\left\{
\begin{aligned}
&U_t=D\Delta U+F(U),~&&t>0,~x\in\mathbb R,\\
&U(0,x)=U_0(x),~&&x\in\mathbb R.
\end{aligned}
\right.
\end{equation}
When $m=2$, traveling wave solutions and entire solutions were obtained for \eqref{1.1} by Li, Xu and Zhang \cite{LXZ2017}, Meng, Yu and Hsu \cite{MYH2019}, and for \eqref{1.5} by Hsu and Yang \cite{HY2013}, Zhao and Wang \cite{ZW2004}, Xu and Zhao \cite{XZ2005}, Wu and Hsu \cite{WH2016}.
When the initial data $U_0$ are compactly supported (or equivalently $U_0(x)\equiv\mathbf 0$ for large $x>0$), there are numerous results on the spreading spread of \eqref{1.1} and \eqref{1.5}. For the nonlocal dispersal system \eqref{1.1}, we refer to Bao et al \cite{BLSW2018}, Bao, Shen and Shen \cite{BSS2019}, Hu et al. \cite{HKLL2015}.
For the local dispersal system \eqref{1.5} and its discrete-time counterpart,
we refer to Kolmogorov, Petrovsky and Piskunov \cite{KPP1937} and Aronson~and~Weinberger~\cite{AW1975,AW1978} for the case $m=1$ (i.e. classical reaction-diffusion equation), and Weinberger \cite{Wei1982}, Lui \cite{Lui1989}, Weinberger, Lewis and Li \cite{WLL2002}, Li, Weinberger and Lewis \cite{LWL2005}, Liang and Zhao \cite{LZ2008,LZ2010}, Fang and Zhao \cite{FZ2014}, and Wang \cite{Wang2011} for the case $m\geqslant2$.

Note that the aforementioned existing results on the spreading speeds of \eqref{1.1} and \eqref{1.5}   essentially assume that the initial data $U_0(x)$ are compactly supported. However, when the initial data $U_0(x)$ are not compactly supported, the results of spreading speed are much fewer. Especially, when the initial value function decays   exponentially, namely
\begin{equation}
u(0,x)\sim C e^{-\sigma|x|}~\text{as}~|x|\rightarrow+\infty~\text{with}~\sigma>0,~ C>0,
\end{equation}
the system \eqref{1.1} with $m=1$, namely \eqref{1.3}, has a spreading speed
\begin{equation}\label{1.4}
s(\sigma)=\frac{1}{\sigma}\bigg\{\int_{\mathbb R}k(x)e^{\sigma x}dx-1+f'(0)\bigg\}~~\text{for}~\sigma\in(0,\sigma^*),
\end{equation}
where $\sigma^*=\min\{\sigma>0~|~s(\sigma)=\min\{s(\sigma);\sigma>0\}\}$, see e.g. \cite{SLW2011,Coville,XLR2020-2}. Similar  results  for   \eqref{1.5} with $m=1$  (i.e. reaction-diffusion equation) and exponentially decaying initial data were previously obtained by Booty,  Haberman and  Minzon \cite{BHM1993}, Hamel and Nadin \cite{HN2012}, McKean \cite{McK1975}, and Sattinger \cite{Sat1976}, etc.
When $m=2$, a recent work  by Xu, Li and Ruan \cite{XLR2020}   studied  the spreading speed of \eqref{1.1}   for initial data $u_{1,0}(x)$ and $u_{2,0}(x)$ decaying  exponentially with the same decay rate.

The purpose of this paper is to study the spreading speed of \eqref{1.1} where $m\geqslant 2$ and all components of   initial data $U_0$ decay  exponentially but  their  decay rates may be different. That is
we assume that  each component of $U_0(x)$ has its own decay  rate, namely
\begin{equation}\label{1.2}
u_{j,0}(x)\sim C_j e^{-\lambda_j|x|}~\text{as}~|x|\rightarrow+\infty~\text{with}~C_j>0~\text{for any}~ j\in J\triangleq\{1,\ldots,m\}.
\end{equation}
We call the initial data $U_0(x)$ are non-uniform if there exist some $i,j\in J$ with $i\ne j$ such that $\lambda_i\ne \lambda_j$.
The case of non-uniform initial data considered in this paper is essentially different from the case in \cite{XLR2020}  where $m=2$ and $\lambda_1=\lambda_2$.
From \eqref{1.4} and  other results mentioned above, we conclude    that the spreading speed of scalar dispersal equations essentially depends on the decay rate of exponentially decaying initial data. For the dispersal system, if all components of initial data $U_0$ have the same decay rate   (i.e. uniform initial data), the spreading speed can still be determined by this single decay rate as shown in \cite{XLR2020} for $m=2$. But now if the initial data are non-uniform, an immediate question is whether all  components of  \eqref{1.1} have the same spreading speed, and if so, which component will paly a prevailing role in determining this spreading speed.
To proceed, we give the  definition  of   uniform spreading speed of \eqref{1.1}.
\begin{definition}[Uniform spreading speed]
Given initial data  $U_0$ satisfying \eqref{1.2}, a positive constant $c_0$  is called the uniform spreading speed of the solution of  \eqref{1.1}, if for any  $j\in J$ and $\varepsilon\in(0,c_0)$, there is a constant $\nu>0$ such that
\[
\left\{
\begin{aligned}
&\lim\limits_{t\rightarrow+\infty}\sup_{|x|\geqslant (c_0+\varepsilon) t}u_j(t,x)=0,\\
&\liminf\limits_{t\rightarrow+\infty}\inf_{|x|\leqslant (c_0-\varepsilon)t}u_j(t,x)\geqslant\nu.
\end{aligned}
\right.
\]
\end{definition}

We will show that when the reaction function $F$ is cooperative and $F'(\mathbf 0)$ is irreducible, all components of the solution $U$ of \eqref{1.1} with non-uniform initial data (different decay rates) satisfying \eqref{1.2}  have a uniform spreading speed (the same spreading speed), see Theorem \ref{th1.3}.
Furthermore, this uniform spreading speed depends only on the smallest decay rate $\lambda_0\triangleq\min \{\lambda_j,j\in J\}$ and is decreasing with respect to $\lambda_0$, which implies  that the component with the smallest decay rate can accelerate the spatial propagation of other components of $U$ (see details in Section 2). We also refer to a recent work by Xu, Li, and Ruan \cite{XLR2017} where the acceleration  propagation of \eqref{1.5} was obtained for non-uniform non-exponentially decaying initial data, and other works by Coulon and Yangari \cite{CY2017}, Yangari \cite{Yan2016}, and  Xu, Li and Lin \cite{XLL2020} for the acceleration  propagation with non-uniform nonlocal dispersal kernels and  compactly supported initial data.

The rest of this paper is organized as follows. In  Section 2, we present  the main  assumptions and results.
In Section  3, we study a special case  where all  components  of initial data have the same decay rate $\lambda$, and prove that \eqref{1.1} has a uniform spreading speed dependent on $\lambda$. In Section 4, we focus on the general case that  the initial data satisfy  \eqref{1.2} and complete the proof of our main result.

\section{Main assumptions and results}
In this section, we give the main  assumptions and results.  Let us introduce  some  notations first. For $U=(u_1,\ldots,u_m)\in \mathbb R^m$, $V=(v_1,\ldots,v_m)\in \mathbb R^m$, we write $U\geqslant V$ if $u_j\geqslant v_j$ for any $j\in J$; $U\gg V$ if $u_j> v_j$ for any $j\in J$.
Denote
\[
[U,V]=\{\phi\in\mathbb R^m; U\leqslant\phi\leqslant V\}.
\]
Let $\|U\|=\sqrt{u_1^2+\ldots+u_m^2}$ denote the    norm of $\mathbb R^m$.
We write $\mathbf 0=(0,\ldots,0)\in \mathbb R^m$ and $\mathbf 1=(1,\ldots,1)\in \mathbb R^m$.
Assume that
\begin{itemize}
\item[\bf (A1)]
\begin{description}
       \item[(a)] there is a strictly positive equilibrium $P=(p_1,p_2,..., p_m)$  such that  $F(\mathbf{0})=F(P)=\mathbf{0}$ and $F\in C^1([\mathbf{0},P],\mathbb{R}^m)$;  there is no other equilibrium $\phi$ in $[\mathbf 0,P]$ such that $F(\phi)=\mathbf{0}$.
       \item[(b)] $F$ is cooperative in $[\mathbf{0},P]$, namely $\frac{\partial}{\partial u_i} f_j(U)\geqslant0 $ for any $U\in[\mathbf{0},P]$ and $j\neq i$.
       \item[(c)]$F'(\mathbf{0})$ is an irreducible  matrix  satisfying
           \[\max\{\text{Re}~\lambda | \det(\lambda I-F'(\mathbf{0}))=0\}>0.\]
       \item[(d)] for any $j\in J$, the function  $k_j$ is nonnegative, continuous, symmetric on $\mathbb R$, and decreasing on $\mathbb R^+$. Moreover,   $\int_\mathbb{R}k_j(x)dx=1$ and there exists  $\Lambda>0$  such that  \begin{equation}\label{1.99}
           \int_\mathbb{R}k_j(x)e^{\Lambda|x|}dx<+\infty.
           \end{equation}
\end{description}
\end{itemize}
Note that  \eqref{1.1} is monostable on $[\mathbf 0, P]$ under  (A1)(a) and (c); namely, the equilibrium  $U\equiv\mathbf 0$ is unstable and $U\equiv P$ is  stable.
From (A1)(b), the matrix $F'(\mathbf 0)$ is essentially nonnegative.
Note that a matrix $A=(a_{ij})_{m\times m}$ is called essentially nonnegative if all   coefficients of the matrix $(A-\min_{i\in\{1,\ldots,m\}}\{a_{ii}\}\mathbf I_m)$ are nonnegative.

We define
\[
\overline \Lambda=\sup\left\{\lambda>0~\Big|~\int_\mathbb R k_j(x)e^{\lambda x}dx<+\infty~\text{for all}~j\in\{1,\ldots,m\}\right\}\in\mathbb (0,+\infty)\cup \{+\infty\}.
\]
For $\lambda\in(0,\overline\Lambda)$, let $\mathcal K(\lambda)$ denote the $m\times m$  matrix as follows
\[
\mathcal K(\lambda)= D\cdot\text{diag}\left\{\int_{\mathbb R}k_1(y)e^{\lambda y}dy,\ldots,\int_{\mathbb R}k_m(y)e^{\lambda y}dy\right\}-D+F'(\mathbf 0).
\]
Since  $F'(\mathbf0)$ is irreducible, so is $\mathcal K(\lambda)$. By the Perron-Frobenius theorem (see \cite{LT1985}), $\mathcal K(\lambda)$ has an eigenvalue $\gamma(\lambda)$ with algebraic multiplicity one, and we denote  by $V(\lambda)$ the positive unit eigenvector  corresponding to  $\gamma(\lambda)$, namely $\mathcal K(\lambda) V(\lambda)=\gamma(\lambda)V(\lambda)$ and
\begin{equation}\label{1.98}
 V(\lambda)\gg\mathbf 0\quad\text{for} ~\lambda\in(0,\overline\Lambda).
\end{equation}
From the  symmetry   of $k_j$, it follows that $\int_{\mathbb R}k_j(y)e^{\lambda y}dy\geqslant1$ for any $\lambda\in(0,\overline \Lambda)$.
Then (A1)-(c)  implies that $\gamma(\lambda)>0$.
For $\lambda\in(0,\overline\Lambda)$,   denote
\begin{equation}\label{1.21}
c(\lambda)=\gamma(\lambda)/\lambda>0.
\end{equation}
Obviously, $c(\lambda)$ is continuous on $(0,\overline\Lambda)$ and
\begin{equation}\label{2.2}
c(\lambda)\lambda V(\lambda)-\mathcal K(\lambda)V(\lambda)=\mathbf 0~\text{for any}~\lambda\in(0,\overline\Lambda).
\end{equation}
Define
\begin{equation}\label{1.22}
c^*\triangleq\inf_{\lambda\in(0,\overline\Lambda)}\{c(\lambda)\}<+\infty.
\end{equation}
It was shown in \cite[Lemma 2.4]{HKLL2015} that $\lambda^*<+\infty$, where $\lambda^*$ is the smallest   positive number at which the above infimum is attained, namely
\[
c^*=c(\lambda^*)=\gamma(\lambda^*)/\lambda^*>0.
\]

\begin{remark}\label{re1.2}\rm
The function  $c(\cdot)$   defined by \eqref{1.21} is strictly decreasing on $(0,\lambda^*)$.
Indeed, by  Lemma 6.5 and (6.5)   in Lui \cite{Lui1989},  $c(\lambda)$ is twice continuously differentiable and decreasing (i.e. $c'(\lambda)\leqslant0$) on $(0,\lambda^*)$, and it  satisfies
\[
(\lambda^2c')'=2\lambda c'+\lambda^2 c''\geqslant0.
\]
Then   $c''(\lambda)\geqslant0$ for $\lambda\in(0,\lambda^*)$.
Suppose that $c(\lambda)$ is decreasing but not strictly decreasing on $(0,\lambda^*)$. Then there exists $\mu\in(0,\lambda^*)$ such that $c'(\mu)=0$.
From $c'(\lambda)\leqslant 0$ and $c''(\lambda)\geqslant0$ on $(0,\lambda^*)$,  we get that $c'(\lambda)=0$ for any $\lambda\in[\mu,\lambda^*)$, which implies by the continuity of $c(\lambda)$ that $c(\lambda)=c(\lambda^*)$ for $\lambda\in[\mu,\lambda^*]$.
On the other hand,
recall  that $\lambda^*$ is the smallest   positive number at which $\inf_{\lambda>0}\{c(\lambda)\}$ is attained, which means $c(\lambda)>c(\lambda^*)$ for $\lambda\in(0,\lambda^*)$. It is a  contradiction.
\end{remark}

There are  some additional  assumptions on $F$.
\begin{description}
\item[(A2)] for $\lambda\in(0,\lambda^*]$, $F(\min\{ P,q V(\lambda)\})\leqslant q F'(\mathbf0) V(\lambda)$ for any $q>0$.
\item[(A3)]  there are positive numbers $q_0$, $\delta_0$, and $M$ such that
    \[
    F(U)\geqslant F'(\mathbf{0})U-MU^{1+\delta_0}~\text{for any}~U\in[\mathbf 0, P]~\text{with}~\|U\|\leqslant q_0,
    \]
    where $U^{1+\delta_0}=(u_1^{1+\delta_0},\ldots,u_m^{1+\delta_0})\in\mathbb R^m$.
\end{description}
The assumptions (A2) and (A3) correspond to the Fisher-KPP assumption  in the scalar case, that is the assumption  $f'(0)u-M u^{1+\delta_0} \leqslant f(u)\leqslant f'(0)u$.
The assumption (A3) can be easily  satisfied, for example,  when  $F\in C^{1+\delta_0}[\mathbf 0,q_0\mathbf 1]$.
As stated in \cite{HKLL2015}, under    (A1)-(A3), $c^*$ is the spreading speed of \eqref{1.1} with  compactly supported initial data.
Denote
\[
\lambda_0\triangleq\min\{\lambda_j~|~j\in J\}.
\]
The following theorem about the  uniform spreading speed for non-uniform initial data is the main result of this paper.

\begin{theorem}\label{th1.3}
Assume (A1), (A2), and (A3) hold.
For  the non-uniform initial data $U_0(x)$  satisfying  \eqref{1.2} with $\lambda_0\in(0,\lambda^*)$, the solution of \eqref{1.1} has a uniform spreading speed $c(\lambda_0)$, which is independent of the decay rate $\lambda_j$ satisfying  $\lambda_j>\lambda_0$. Moreover, $c(\lambda_0)$ is strictly decreasing with respect to $\lambda_0\in(0,\lambda^*)$.
\end{theorem}

From Theorem \ref{th1.3}, the cooperation and irreducibility of reaction functions can ensure that all components of the solution of \eqref{1.1} with non-uniform initial data have a  uniform spreading speed.
In fact, if $F\in C^1[\mathbf 0,P]$ and $\frac{\partial}{\partial u_i}f_j(\mathbf 0)>0$ with $i\neq j$, then as seen from the $j$th equation of \eqref{1.1}, namely
\[
\frac{\partial}{\partial t}u_j=d_j(k_j*u_j-u_j)+f_j(U),
\]
the component $u_i$ of $U$ has a direct positive effect on the growth of the component $u_j$, when  $u_j$ is small enough.
We say $u_i$ has an indirect positive effect on the growth of  $u_j$, if $u_i$ does not directly affect the growth of $u_j$ (i.e. $\frac{\partial}{\partial u_i}f_j(\mathbf 0)=0$), but through   other components of $U$,  in the sense that there exists a set $\{j_1,j_2,\ldots,j_k\}$ with $j_1=i$ and $j_k=j$ such that $\frac{\partial}{\partial u_{j_{p-1}}}f_{j_p}(\mathbf 0)>0$ for any $p=2,\ldots,k$.
The irreducibility of $F'(\mathbf 0)=(\frac{\partial}{\partial u_i}f_j(\mathbf 0))_{m\times m}$ means  that  a direct or indirect positive effect exists between any two components of $U$, and hence, all   components  of the solution  with non-uniform initial data can have a uniform spreading speed.

Theorem \ref{th1.3} shows that the uniform spreading speed  depends only on the smallest decay rate $\lambda_0$. This conclusion, along with the fact that the spreading speed $c(\lambda_0)$ is strictly decreasing on $(0,\lambda^*)$ in Remark \ref{re1.2}, means that  the component with the smallest decay rate can accelerate the spatial propagation of other components. To understand this, we assume
the $j_0$th component of initial data $U_0$ has the smallest decay rate, namely $\lambda_{j_0}=\lambda_0\in(0,\lambda^*)$.
Let the decay rate of the $j_0$th component $u_{j_0,0}$   become smaller and  fix  the decay rates of other components of initial data $U_0$.
We denote the new decay rate of $u_{j_0,0}$   by $\lambda'\in(0,\lambda_0)$.
Then the uniform spreading speed   becomes $c(\lambda')$ from $c(\lambda_0)$.
Since $c(\cdot)$  is strictly decreasing on $(0,\lambda^*)$, we have that $c(\lambda')>c(\lambda_0)$, which means that the decrease of the smallest decay rate in the initial data can increase  the spreading speed  of other components of the solution.

Our idea  to prove Theorem \ref{th1.3} consists of two steps. First,  we focus  on  the special case that all components of initial data $U_0$  have the same decay  rate $\lambda\in (0,\lambda^*)$ (namely $\lambda_j=\lambda$ for any $j\in J$) and prove that the solution  has a uniform spreading speed $c(\lambda)$ in Section 3.
Second,  the general case that  $U_0$ satisfies \eqref{1.2} with $\lambda_0\in(0,\lambda^*)$ is considered in Section 4. By constructing  a lower solution, we show that after a period of time $T>0$, all   components of  $U(T,\cdot)$
are larger than an exponentially decaying  function with the decay rate $\lambda_0$.
This case is then transformed into the special case considered in Section 3 as long as   $u_j(T,x)$ is set as the new initial data.

Moreover,  from Theorem \ref{th1.3} and its proof in Section 4, the components whose  decay rates are not  $\lambda_0$ affect neither the result of uniform spreading speed nor its proof method.
Therefore, Theorem \ref{th1.3} also holds  if \eqref{1.2} is changed by the following assumption
\begin{description}
\item[(H)] there exist  $j_0\in\{1,2,\ldots,m\}$  and $\lambda_0>0$ such that
    \[
    u_{j_0,0}(x)\sim Ce^{-\lambda_0|x|},~ u_{j,0}(x)\leqslant e^{-\lambda_0|x|}~\text{for}~j\neq j_0~\text{and}~|x|~\text{large enough}.
    \]
\end{description}
In this assumption,  the component $u_{j,0}$ of $U_0$ with $j\neq j_0$  is not restricted to exponentially decaying functions, but any function
that is smaller than $e^{-\lambda_0|x|}$ when $|x|$ is large enough.

\begin{remark}\rm The   methods in this paper are also applicable to the reaction-diffusion cooperative system \eqref{1.5}. Therefore, no matter whether we consider a nonlocal or local dispersal system, the cooperation and irreducibility of $F$ can ensure that the solution  has a uniform  spreading speed  and   the component of $U$ with the smallest decay rate can accelerate the spatial propagation of other components.
\end{remark}

\section{Case of the same decay  rate}
In this section, we consider the case that  all components of initial data have the same decay rate  $\lambda\in(0,\lambda^*)$. First, we state two important lemmas that are proved in \cite[Theorem 4.1]{XLR2020} (for Lemma \ref{th1.1}) and \cite[Theorem 4.5]{HKLL2015} (for Lemma \ref{le-cp}).

\begin{lemma}\label{th1.1} (Symmetry and  monotone  property)
If the functions $k_j(\cdot)$ and $u_{j,0}(\cdot)$ are symmetric on $\mathbb R$ and decreasing on $\mathbb R^+$ for any $j\in J$, so is $u_j(t,\cdot)$ for any $t>0$ and $j\in J$.
\end{lemma}

\begin{lemma}\label{le-cp}
(Comparison principle)
Assume that $\bar U$ is an upper solution   and $\underbar U$ is a lower solution of  \eqref{1.1}; namely $\frac{\partial }{\partial t}\bar U(t,x)$ and $\frac{\partial }{\partial t}\underbar U(t,x)$ exist  and
\[
\begin{aligned}
&\frac{\partial}{\partial t}\bar U- DK*\bar U+D\bar U-F(\bar U)\geqslant\mathbf0~\text{for}~t>0,~x\in\mathbb R,\\
&\frac{\partial}{\partial t}\underbar U- DK*\underbar U+D\underbar U-F(\underbar U)\leqslant\mathbf0~\text{for}~t> 0,~x\in\mathbb R.
\end{aligned}
\]
If $\bar U(0,x)\geqslant \underbar U(0,x)$ for $x\in\mathbb R$, then $\bar U(t,x)\geqslant\underbar U(t,x)$ for any $t\geqslant0$ and $x\in\mathbb R$.
\end{lemma}

The following result is a  special case of  Theorem \ref{th1.3} where all components of $U_0$ have the same decay rate  $\lambda\in(0,\lambda^*)$.
\begin{proposition}\label{th2.1}
Assume (A1), (A2), and (A3) hold.
Let $U_0(x)$ satisfy \eqref{1.2} with  $\lambda_j=\lambda\in(0,\lambda^*)$ for any $j\in J$.   Then the solution of  \eqref{1.1} has a uniform spreading speed $c(\lambda)$.
\end{proposition}
\begin{proof}
Let $U=(u_1,\ldots,u_m)$ be the solution of  \eqref{1.1} with initial data $U_0$.
By \eqref{1.2} and \eqref{1.98}, there is a constant $\Gamma>0$ large enough such that
\[
U_0(x)\ll \Gamma e^{-\lambda|x|}V(\lambda).
\]
For $\lambda\in(0,\lambda^*)$, define
\begin{equation}\label{2.97}
\bar U(t,x)=\min\left\{P,~\Gamma e^{-\lambda z}V(\lambda)\right\}\quad\text{with}~z=|x|-c(\lambda)t,~t\geqslant0,~x\in\mathbb R.
\end{equation}
Now we check that $\bar U=(\bar u_1,\ldots,\bar u_m)$ is an upper solution. Let $v_j(\lambda)$ denote the $j$th   component of $V(\lambda)$, namely  $V(\lambda)=(v_1(\lambda),\ldots,v_m(\lambda))$. For any $j\in J$, when $z<\lambda^{-1}\ln (\Gamma v_j(\lambda) / p_j)$, we have that $\bar u_j(t,x)=p_j$. Then by (A1)-(b), from $\bar u_i(t,x)\leqslant p_i$ for any $i\in J$ we can get that
\[
\frac{\partial}{\partial t}\bar u_j- d_jk_j*\bar u_j+d_j\bar u_j-f_j(\bar U) \geqslant-f_j( P)=0.
\]
When $z\geqslant\lambda^{-1}\ln (\Gamma v_j(\lambda) / p_j)$, it holds that  $\bar u_j(t,x)=\Gamma e^{-\lambda z}v_j(\lambda)$.
We denote $f_{j,i}=\frac{\partial}{\partial u_i}f_j(\mathbf 0)$ and  (A1)-(b) implies $f_{j,i}\geqslant0$ for $i\neq j$.
By (A2) and \eqref{2.2}, we  have that
\[
\begin{aligned}
&\frac{\partial}{\partial t}\bar u_j- d_jk_j*\bar u_j+d_j\bar u_j-f_j(\bar U) \\
\geqslant& \Gamma e^{- \lambda z}\Big[\Big(c(\lambda)\lambda - d_j\int_{\mathbb R}k_j(y)e^{\lambda y}dy+d_j\Big)v_j(\lambda)-\sum\limits_{i=1}^{m} f_{j,i}v_i(\lambda)\Big]=0.
\end{aligned}
\]
Thus $\bar U=(\bar u_1,\ldots,\bar u_m)$ is an upper solution of  \eqref{1.1}.
Lemma \ref{le-cp} implies that
\[
U(t,x)\leqslant\bar U(t,x)\leqslant \Gamma e^{-\lambda z}V(\lambda)~\text{for any}~ t\geqslant0~\text{and}~x\in\mathbb R.
\]
Then for any $\varepsilon>0$ and $j\in J$, we have that
\[
\lim\limits_{t\rightarrow+\infty}\sup_{|x|\geqslant (c(\lambda)+\varepsilon) t} u_j(t,x) \leqslant
\lim\limits_{t\rightarrow+\infty}\sup_{|x|\geqslant (c(\lambda)+\varepsilon)t}\Gamma e^{-\lambda(|x|-c(\lambda)t)}v_j(\lambda)\leqslant\lim\limits_{t\rightarrow+\infty}\Gamma e^{-\lambda\varepsilon t}v_j(\lambda)=0.
\]

Now we just need to  prove that for any   $\varepsilon\in(0,c(\lambda))$ and $j\in J$, there exists $\nu>0$ such that
\begin{equation} \label{2.5}
\lim\limits_{t\rightarrow+\infty}\inf_{|x|\leqslant (c(\lambda)-\varepsilon)t}u_j(t,x)\geqslant\nu.
\end{equation}
The proof of \eqref{2.5} consists of  the following two steps.

First, we prove that there exist   two positive constants   $\gamma$ and $y_0$ such that
\begin{equation}\label{2.9}
U(1,x)\geqslant \gamma\min\left\{e^{-\lambda|x|},e^{-\lambda y_0}\right\}V(\lambda),~x\in\mathbb R.
\end{equation}
From \eqref{1.2} it follows that $U_0(x)\gg \mathbf 0$ for sufficiently large $|x|$. Then by (A1)-(d), there exists $N_0\in\mathbb N^+$ such that
\begin{equation}\label{3.99}
\underbrace{K*K*\ldots*K}\limits_{N_0}*U_0(x)\gg\mathbf 0~\text{for any}~x\in\mathbb R.
\end{equation}
For $j\in J$, let $\pi_j:\mathbb R^m\rightarrow \mathbb R^m$ denote the function
\[
\pi_j: (u_1,\ldots,u_m)\mapsto(0,\ldots,u_j,\ldots,0);
\]
namely the $j$th component of $\pi_j(U)$ is $u_j$ while others are zero. We define
\begin{equation}\label{2.7}
b_j=\inf\limits_{u_j\in(0,p_j]} \{f_j(\pi_j(U))/u_j \}.
\end{equation}
Let $n\in\mathbb N^+$ and we divide equally  the time period of $[0,\tau]$  into $n$ parts, namely $[0,\tau/n]$, $[\tau/n, 2\tau/n]$,$\ldots$, and $[(n-1)\tau/n, \tau]$.
In $[0,\tau/n]$, we consider
\[
\underline{W} (t,x)=(\underline {w}_1(t,x),\dots,\underline {w}_m(t,x)),~t\in [0,\tau/n],~x\in\mathbb R,
\]
where
\begin{equation}\label{2.96}
\underline {w}_j(t,x)=M_j\big[u_{j,0}(x)+td_jk_j*u_{j,0}(x)\big]e^{(b_j-d_j)t},~j\in J\\
\end{equation}
and
\[
M_j=(1+d_j\tau/n)^{-1}(1+e^{(b_j-d_j)\tau/n})^{-1},~j\in J.
\]
It is easy to check that
\[
\partial_t\underline{w}_j-d_jk_j*\underline w_j+d_j\underline w_j-b_j\underline w_j\leqslant0\quad\text{for}~j\in J.
\]
For $t\in[0,\tau/n]$,  by $u_{j,0}(x)\leqslant p_j$ we have that
\[
\underline {w}_j(t,x)\leqslant M_j p_j[1+d_j\tau/n]e^{(b_j-d_j)\tau/n}\leqslant p_j \quad\text{for}~x\in\mathbb R.
\]
From (A1)-(b)  and \eqref{2.7}, it follows  that
\[
\begin{aligned}
&\underline W_t-DK*\underline W+D\underline W-F(\underline W)\\
\leqslant &~\underline W_t-DK*\underline W+D\underline W-(f_1(\pi_1(\underline W)),\ldots,f_m(\pi_m(\underline W)))\\
\leqslant &~\underline W_t-DK*\underline W+D\underline W-\text{diag}\{b_1,\ldots,b_m\}\underline W\leqslant\mathbf 0.
\end{aligned}
\]
By $\underline{W} (0,x)\leqslant U_0(x)$ for $x\in\mathbb R$,  from  Lemma \ref{le-cp} we get that
\begin{equation}\label{2.8}
U(\tau/n,x)\geqslant \underline W(\tau/n,x).
\end{equation}
Denote $C_1\triangleq\min\limits_{j\in J}\{M_j d_j e^{(b_j-d_j)\tau/n}\tau/n\}$ and then
\[
U(\tau/n,x)\geqslant C_1 K*U_0(x).
\]
Repeat this argument for  $t\in[\tau/n,2\tau/n]$ and substitute $K*U_0(x)$ for  $U_0(x)$. We can find a constant  $C_2>0$ such that
\[
U(2\tau/n,x)\geqslant C_2 K*K*U_0(x).
\]
Similarly, there exists  $C_n>0$ such that
\[
U(\tau,x)\geqslant C_n \underbrace{K*K*\ldots*K}\limits_{n}*U_0(x)~\text{for}~x\in\mathbb R.
\]
When $n=N_0$, it follows from \eqref{3.99} that $U(\tau,x)\gg\mathbf 0$.
When  $n=1$, we get   from \eqref{2.96} and \eqref{2.8}   that
\[
U(\tau,x)\geqslant \underline W(\tau,x)\geqslant C_\tau U_0(x)~\text{with}~C_\tau=\min\limits_{j\in J}\{M_j e^{(b_j-d_j)\tau}\}.
\]
Then for any $\tau>0$ there exists $C_\tau>0$ such that
\begin{equation}\label{2.6}
U(\tau,x)\gg\mathbf 0~\text{and}~U(\tau,x)\geqslant C_\tau U_0(x)~\text{for}~x\in\mathbb R.
\end{equation}
When $\tau=1$,  by \eqref{1.2} with $\lambda_j=\lambda$ we can find  $\gamma>0$ and $y_0>0$ satisfying \eqref{2.9}.

Let $\gamma$ be smaller (if necessary) such that  $\gamma e^{-\lambda y_0}\leqslant q_0$, where $q_0$ is given by assumption (A3).
Define $W_0(x)=\gamma\min\left\{e^{-\lambda|x|},e^{-\lambda y_0}\right\}V(\lambda)$, $x\in\mathbb R$. Then $\|W_0(x)\|\leqslant q_0$ for $x\in\mathbb R$ and
\begin{equation}\label{2.9}
U(1,x)\geqslant W_0(x)=
\left\{
\begin{aligned}
&\gamma e^{-\lambda|x|}V(\lambda)~&&\text{for}~|x|\geqslant y_0,\\
&\gamma e^{-\lambda y_0}V(\lambda)~&&\text{for}~|x|\leqslant y_0,
\end{aligned}
\right.
\end{equation}
Let $W(t,x)$ be the solution of \eqref{1.1} with initial data $W(0,x)=W_0(x)$. Then we get from Lemma \ref{le-cp} that
\begin{equation}\label{2.10}
U(t+1,x)\geqslant W(t,x)~\text{for}~t\geqslant 0,~x\in\mathbb R.
\end{equation}
Since $W_0(\cdot)$ is symmetric  and decreasing on $\mathbb R^+$, so is $W(t,\cdot)$ by Lemma \ref{th1.1}.

Second,  we construct a lower solution  and prove \eqref{2.5}.
Now define some nations. By Remark \ref{re1.2}, for any $\lambda\in(0,\lambda^*)$, there is a constant $\delta_\lambda=\lambda^*/\lambda-1>0$ such that
\[
c(\lambda+\lambda s)<c(\lambda)~\text{for any}~s\in(0,\delta_\lambda).
\]
Denote
\[
\mu=\lambda(1+\delta)>0\quad\text{with}\quad\delta\triangleq\min\{\delta_0,\delta_\lambda/2\}>0,
\]
where the positive constant $\delta_0$ is given by (A3). Then it follows that
\begin{equation}\label{2.11}
c(\mu)< c(\lambda).
\end{equation}
For $j\in J$, let $G(c,\lambda;j)$ be the $j$th component of the vector $c\lambda V(\lambda)-\mathcal K(\lambda)V(\lambda)$; namely
\[
G(c,\lambda;j)\triangleq\left(c\lambda - d_j\int_{\mathbb R}k_j(y)e^{\lambda y}dy+d_j\right)v_j(\lambda)-\sum\limits_{i=1}^{m} f_{j,i}v_i(\lambda),\quad c>0,~\lambda>0,
\]
where $f_{j,i}=\frac{\partial}{\partial u_i}f_j(\mathbf 0)$ and $v_j(\lambda)$ is the $j$th   component of $V(\lambda)\gg\mathbf 0$.
For $\lambda\in(0,\lambda^*)$, it follows from \eqref{2.2} that
\begin{equation}\label{2.12}
G(c(\lambda),\lambda;j)=\left(c(\lambda)\lambda - d_j\int_{\mathbb R}k_j(y)e^{\lambda y}dy+d_j\right)v_j(\lambda)-\sum\limits_{i=1}^{m} f_{j,i}v_i(\lambda)=0.
\end{equation}
By \eqref{2.11} we get that
\begin{equation}\label{2.13}
\begin{aligned}
G(c(\lambda),\mu;j)&=\left(c(\lambda)\mu - d_j\int_{\mathbb R}k_j(y)e^{\mu y}dy+d_j\right)v_j(\mu)-\sum\limits_{i=1}^{m} f_{j,i}v_i(\mu)\\
&>G(c(\mu),\mu;j)=0.
\end{aligned}
\end{equation}

For $\lambda\in(0,\lambda^*)$, we define  $\underline U=(\underline u_1,\ldots,\underline u_m)$ as follows
\[
\underline U(t,x)=\max\left\{\mathbf 0,~\gamma e^{- \lambda z}V(\lambda)-Le^{-\mu z}V(\mu)\right\}~\text{with}~z=|x|-c(\lambda)t,~t\geqslant0,~x\in\mathbb R,
\]
where $L$ is a positive constant  large enough such that
\begin{equation}\label{2.14}
L\geqslant\max\left\{\frac{\gamma e^{\lambda\delta y_0}}{1+\delta}\max\limits_{j\in J}\left(\frac{v_j(\lambda)}{v_j(\mu)}\right),~M\gamma^{1+\delta}\max\limits_{j\in J}\left(\frac{v_j^{1+\delta}(\lambda)}{G(c(\lambda),\mu;j)}\right)\right\}.
\end{equation}
Denote
\[
y_j=\lambda^{-1}\delta^{-1}\ln\left(\frac{L(1+\delta)v_j(\mu)}{\gamma v_j(\lambda)}\right)~\text{and}~ z_j=\lambda^{-1}\delta^{-1}\ln\left(\frac{Lv_j(\mu)}{\gamma v_j(\lambda)}\right)\quad\text{for}~ j\in J.
\]
Then $y_j>z_j$ for any $j\in J$.
Note that $y_j$ and $z_j$ correspond respectively to the maximum point  of    $z\mapsto \gamma e^{- \lambda z}v_j(\lambda)-Le^{-\mu z}v_j(\mu)$ and the root of  $\gamma e^{- \lambda z}v_j(\lambda)-Le^{-\mu z}v_j(\mu)=0$,  that is
\begin{equation}\label{3.96}
\underline u_j(t,x)=\left\{
\begin{aligned}
&0,&&\text{when}~z<z_j,\\
&\gamma e^{- \lambda z}v_j(\lambda)-Le^{-\mu z}v_j(\mu)=0,&&\text{when}~z=z_j,\\
& \gamma e^{- \lambda z}v_j(\lambda)-Le^{-\mu z}v_j(\mu)>0,&&\text{when}~z>z_j,\\
\end{aligned}
\right.
\end{equation}
and
\[
\max\limits_{z\in\mathbb R}\{\gamma e^{- \lambda z}v_j(\lambda)-Le^{-\mu z}v_j(\mu)\}=\gamma e^{- \lambda y_j}v_j(\lambda)-Le^{-\mu y_j}v_j(\mu)>0.
\]
From \eqref{2.14}, it follows  that $y_j\geqslant y_0$ for any $j\in J$. Then we have that
\begin{equation}\label{2.16}
\sup\limits_{t\geqslant0,x\in\mathbb R}\underline u_j(t,x)=\underline u_j(t,c(\lambda)t+y_j)=\gamma e^{-\lambda y_j}v_j(\lambda)-Le^{-\mu y_j}v_j(\mu)\leqslant \gamma e^{-\lambda y_0}v_j(\lambda).
\end{equation}
Since $V(\lambda)$ is a  unit vector, it holds that
\begin{equation}\label{2.18}
\|\underline U(t,x)\|\leqslant  \gamma e^{-\lambda y_0}\leqslant q_0\quad\text{for any}~t\geqslant0,~x\in\mathbb R.
\end{equation}
Particularly, when $t=0$, it follows  from \eqref{2.16} that $\underline U(0,x)\leqslant \gamma e^{-\lambda y_0}V(\lambda)$ for  any $x\in\mathbb R$. The definition of $\underline U(t,x)$ implies  that  $\underline U(0,x)\leqslant\gamma e^{-\lambda|x|}V(\lambda)$ for   $x\in\mathbb R$.
Then we get from  \eqref{2.9} that
\begin{equation}\label{2.17}
  \underline U(0,x)\leqslant W_0(x)~\text{for}~x\in\mathbb R.
\end{equation}

In order to verify $\underline U(t,x)$ is a lower solution, namely
\[
\underline U_t- DK*\underline U+D\underline U-F(\underline U)\leqslant\mathbf0,
\]
we check it holds for each  component.
For any $j\in J$, when  $z<z_j$, since $\underline u_j(t,x)= 0$, it is easy to check that
\[
\frac{\partial}{\partial_t}\underline u_j - d_jk_j*\underline u_j+d_j\underline u_j-f_j(\underline U)\leqslant 0.
\]
When $z\geqslant z_j$, we get that
\[
\begin{aligned}
&\underline u_j(t,x)=\gamma e^{-\lambda z}v_j(\lambda)-Le^{-\mu z}v_j(\mu),\\
&\underline u_i(t,x)\geqslant\gamma e^{-\lambda z}v_i(\lambda)-Le^{-\mu z}v_i(\mu)~\text{for}~ i\neq j.
\end{aligned}
\]
From (A3) and \eqref{2.18}, it follows that
\[
\begin{aligned}
f_j(\underline U)
&\geqslant \sum\limits_{i=1}^{m} f_{j,i}\underline u_i(t,x)-M\underline u_j^{1+\delta}(t,x)\\
&\geqslant \sum\limits_{i=1}^{m} f_{j,i}\left[\gamma e^{-\lambda z}v_i(\lambda)-Le^{-\mu z}v_i(\mu)\right]-M\gamma^{1+\delta} e^{-\mu z}v_j^{1+\delta}(\lambda).
\end{aligned}
\]
Then some calculations show that
\[
\begin{aligned}
&\frac{\partial}{\partial t}\underline u_j- d_jk_j*\underline u_j+d_j\underline u_j-f_j(\underline U)\\
\leqslant&\gamma e^{- \lambda z}\Big[\Big(c(\lambda)\lambda - d_j\int_{\mathbb R}k_j(y)e^{\lambda y}dy+d_j\Big)v_j(\lambda)-\sum\limits_{i=1}^{m} f_{j,i}v_i(\lambda)\Big]\\
&\quad-L e^{-\mu z}\Big[\Big(c(\lambda)\mu - d_j\int_{\mathbb R}k_j(y)e^{\mu y}dy+d_j\Big)v_j(\mu)-\sum\limits_{i=1}^{m} f_{j,i}v_i(\mu)\Big]+M\gamma^{1+\delta} e^{-\mu z}v_j^{1+\delta}(\lambda)\\
=&\gamma e^{-\lambda z}G(c(\lambda),\lambda;j)- e^{-\mu z}\left[LG(c(\lambda),\mu;j)-M\gamma^{1+\delta} v_j^{1+\delta}(\lambda)\right].
\end{aligned}
\]
By \eqref{2.12}, \eqref{2.13}, and \eqref{2.14}, for $z\geqslant z_j$, we have
\[
\frac{\partial}{\partial t}\underline u_j- d_jk_j*\underline u_j+d_j\underline u_j-f_j(\underline U)\leqslant0.
\]
Therefore, $\underline U(t,x)$ is a lower solution.

Lemma \ref{le-cp} and   \eqref{2.17} imply that
\[
W(t,x)\geqslant \underline U(t,x)~\text{for any}~t\geqslant0,~x\in\mathbb R.
\]
Let $y_{\text{max}}\triangleq\max\limits_{j\in J}\{y_j\}$.  It follows from $y_j>z_j$ that $y_{\text{max}}>\max\limits_{j\in J}\{z_j\}$, which implies by \eqref{3.96} that
\[
\nu\triangleq\min\limits_{j\in J}\{\underline u_j(t,c(\lambda)t+y_{\text{max}})\}>0.
\]
We denote $W(t,x)$ by $(w_1(t,x),\ldots,w_m(t,x))$. Then  it follows   that
\[
 w_j(t,c(\lambda)t+y_{\text{max}})\geqslant\underline  u_j(t,c(\lambda)t+y_{\text{max}})\geqslant\nu~\text{for any}~t\geqslant0~\text{and}~j\in J.
\]
Since $W(t,\cdot)$ is   symmetric  and decreasing on $\mathbb R^+$, it holds that
\[
w_j(t,x)\geqslant \nu ~\text{for any}~|x|\leqslant c(\lambda)t+y_{\text{max}}~\text{and}~j\in J.
\]
By \eqref{2.10} we get that
\[
u_j(t+1,x)\geqslant \nu ~\text{for any}~|x|\leqslant c(\lambda)t+y_{\text{max}}~\text{and}~j\in J,
\]
which implies \eqref{2.5}.  This completes the proof of Proposition \ref{th2.1}.
\end{proof}

\section{General case}

In this section, we give the proof of  Theorem \ref{th1.3}. By constructing a lower solution,  we transform the proof for the  general case where $U_0$ satisfies \eqref{1.2} into the   special cases in Section  3, where all components of the initial data have the same decay rate.

\begin{proof}[Proof of Theorem \ref{th1.3}]
The strictly decreasing property of $c(\lambda_0)$  with respect to $\lambda_0\in(0,\lambda^*)$ has been  obtained in Remark \ref{re1.2}.
By \eqref{1.2} and $\lambda_0\leqslant\lambda_j$ for $j\in J$,
there exists $C>0$ such that $u_{j,0}(x)\leqslant C e^{-\lambda_0|x|}$ for $j\in J$ and large $|x|$.
Then the proof of
\[
\lim\limits_{t\rightarrow+\infty}\sup_{|x|\geqslant (c(\lambda_0)+\varepsilon) t}u_j(t,x)=0\quad\text{for}~j\in J
\]
is similar to the counterpart in the proof of Propositions \ref{th2.1}, and we only need to substitute $\lambda_0$ for $\lambda$.

Now prove that for any  $\varepsilon\in(0,c(\lambda_0))$, there is a constant $\nu>0$ such that
\[
\lim\limits_{t\rightarrow+\infty}\inf_{|x|\leqslant (c(\lambda_0)-\varepsilon)t}u_j(t,x)\geqslant\nu~\text{for any}~ j\in J.
\]
From the proof  of  Propositions \ref{th2.1},   we only need to prove that there exist  $T>0$ and $M_0>0$ such that
\begin{equation}\label{4.1}
u_j(T,x)\geqslant M_0 p(x),~x\in\mathbb R\quad~\text{for any}~ j\in J,
\end{equation}
where
\[
p(x)=e^{-\lambda_0|x|}.
\]

Now we reorder the equations in the system \eqref{1.1} (namely, reorder the components of $U$).
Define
\[
f_{j,i}=\frac{\partial}{\partial u_i}f_j(\mathbf 0).
\]
Choose the component who has the smallest decay rate as the first component $u_1$ of $U$, and then $\lambda_1=\lambda_0=\min\{\lambda_j,j\in J\}$. Since $F'(0)$ is irreducible, we can choose   the second  component $u_2$ such that $f_{2,1}>0$.  Similarly, we can choose the third  component   $u_3$ satisfying $f_{3,1}>0$ or $f_{3,2}>0$.
Repeat this process, we   reorder the components of $U$ satisfying that
for any $i\in \{2,3,\ldots,m\}$ there exists $j\in\{1,2,\ldots,i-1\}$ such that $f_{i,j}>0$.

We give   an important inequality. Since  $F\in C^1[\mathbf{0},P])$, by (A1)-(b), we can find  a constant  $q_3>0$ such that
\begin{equation}\label{4.2}
f_j(U)\geqslant (f_{j,j}-1)u_j+\frac{1}{2}\sum\limits_{i\neq j}f_{j,i}u_i~~~\text{for any}~j\in J~\text{and}~ U\in [\mathbf 0, q_3\mathbf 1].
\end{equation}

In order to prove \eqref{4.1}, we need to construct a lower solution
\[
W(t,x)=(w_1(t,x),\ldots,w_m(t,x))\in [\mathbf 0, q_3\mathbf 1],\quad t\geqslant1,~x\in\mathbb R.
\]
The form of $W(t,x)$ will be given for every component.
First, we construct the first component $w_{1}(t,x)$ of  $W(t,x)$. By \eqref{2.6} and  \eqref{1.2} with $\lambda_1=\lambda_0$, there is  a constant $C_0\in(0,q_3]$ such  that
\[
u_{1}(1,x)\geqslant C_0 p(x)\quad~\text{for}~x\in\mathbb R.
\]
Let
\[
w_{1}(t,x)= M_1 e^{-\alpha (t-1)}p(x)~~\text{for}~t\geqslant 1,~x\in\mathbb R,
\]
where $M_1$ is a constant  in $(0,C_0]$  and
\[
\alpha\geqslant\max\limits_{j\in J}\{d_j+|f_{j,j}|\}+2.
\]
Note that $M_1$ will be reselect as a smaller constant later.
It is easy to check that
\[
w_{1}(t,x)\leqslant M_1 \leqslant C_0\leqslant q_3\quad\text{for}~t\geqslant 1,~x\in\mathbb R
\]
and
\begin{equation}\label{4.3}
w_{1}(1,x)\leqslant M_1p(x)\leqslant u_{1}(1,x)~\text{for}~x\in \mathbb R.
\end{equation}
From $p(x)\geqslant0$, it follows that  $k_{1}*w_{1}\geqslant0$. By   the cooperation of $F$ and \eqref{4.2}, we have that $f_{1}(W)\geqslant (f_{1,1}-1)w_{1}$ for  $W\in [\mathbf 0, q_3\mathbf 1]$.  Then some calculations show that
\[
\begin{split}
& \frac{\partial}{\partial t} w_{1}-d_{1}k_{1}*w_{1}+d_{1}w_{1}-f_{1}(W)\\
\leqslant&~M_1(-\alpha+d_{1}-f_{1,1}+1)e^{-\alpha(t-1)}p(x)\leqslant 0.
\end{split}
\]

Second, we construct the second component $w_{2}(t,x)$ of  $W(t,x)$ under the condition  $f_{2,1}>0$.
Define
\[
w_{2}(t,x)=M_2\left(e^{-\beta_2(t-1)}-e^{-\alpha (t-1)}\right)p(x)\quad\text{for}~t\geqslant 1,~x\in\mathbb R,
\]
where
\[
\beta_2=d_{2}+|f_{2,2}|+1,\quad M_2\triangleq\frac{f_{{2},{1}}M_1}{2(\alpha-d_{2}+f_{2,2}-1)}.
\]
By  $\alpha\geqslant\beta_2$, we get that $w_{2}\geqslant0$ for $t\geqslant1$,  which implies that $k_{2}*w_{2}\geqslant0$.
Let $M_1$ be smaller (if necessary) satisfying  $M_1\leqslant 2q_3/f_{{2},{1}}$. From $\alpha\geqslant d_{2}-f_{2,2}+2$, it follows that
\[
w_{2}(t,x)\leqslant M_2\leqslant \frac{1}{2}f_{{2},{1}}M_1\leqslant q_3~\text{for}~t\geqslant 1,~x\in\mathbb R.
\]
Assumption  (A1)-(b) and \eqref{4.2} show that $f_{2}(W)\geqslant (f_{2,2}-1)w_{2}+\frac{1}{2}f_{2,1}w_{1}$ for  $W\in [\mathbf 0, q_3\mathbf 1]$.
Then we have  that
\[
\begin{split}
& \frac{\partial}{\partial t} w_{2}-d_{2}k_{2}*w_{2}+d_{2}w_{2}-f_{2}(W)\\
\leqslant&~\frac{\partial}{\partial t} w_{2}-d_{2}k_{2}*w_{2}+d_{2}w_{2}-(f_{2,2}-1)w_{2}-\frac{1}{2}f_{2,1}w_{1}\\
\leqslant&~M_2\left[(-\beta_2+d_{2}-f_{2,2}+1)e^{-\beta_2(t-1)}+(\alpha-d_{2}+f_{2,2}-1)e^{-\alpha(t-1)}\right]p(x)\\
&\quad-\frac{1}{2}M_1f_{2,1}e^{-\alpha(t-1)}p(x)\\
\leqslant&~\Big[M_2(\alpha-d_{2}+f_{2,2}-1)-\frac{1}{2}M_1 f_{2,1}\Big]e^{-\alpha(t-1)}p(x)=0.
\end{split}
\]
Moreover,  it is easy to check that
\[
w_{2}(1,x)=0~\text{for}~x\in\mathbb R.
\]
Note that $e^{-\beta_2s}\geqslant 2e^{-\alpha s}$ for $s\geqslant \tau\triangleq \ln 2$ and then
\begin{equation}\label{4.4}
w_{2}(t,x)\geqslant M_2e^{-\alpha (t-1)}p(x)\quad\text{for}~t\geqslant 1+\tau,~x\in\mathbb R,
\end{equation}
which is a key inequality for the construction of $w_j$ with $j>2$ when $f_{j,2}>0$.

Third, we construct the third  component $w_{3}(t,x)$ of  $W(t,x)$ under the condition  $f_{3,1}>0$ or $f_{3,2}>0$. For the case $f_{3,1}>0$, we can construct  $w_{3}(t,x)$ by the same method as $w_{2}(t,x)$. For the case $f_{3,2}>0$, we define
\[
w_{3}(t,x)=\left\{
\begin{array}{ll}
0,&1\leqslant t\leqslant 1+\tau,\\
M_3\Big(e^{-\beta_3(t-1-\tau)}-e^{-\alpha (t-1-\tau)}\Big)p(x),&t\geqslant 1+\tau,
\end{array}
\right.
\]
where
\[
\beta_3=d_{3}+|f_{3,3}|+1,\quad M_3\triangleq\frac{f_{{3},{2}}M_2}{2e^{\alpha \tau}(\alpha-d_{3} +f_{3,3} -1)}.
\]
Let $M_1$ be  smaller   (if necessary) such that $M_3\leqslant q_3$, and then $0\leqslant w_3(t,x)\leqslant q_3$ for $t\geqslant 1$,~$x\in\mathbb R$.
By \eqref{4.2} and \eqref{4.4}, we have that
\[
f_{3}(W)\geqslant (f_{3,3}-1)w_3+\frac{f_{3,2}}{2}w_2\geqslant (f_{3,3}-1)w_3+\frac{f_{3,2}}{2} M_2e^{-\alpha (t-1)}p(x)\quad\text{for}~W\in [\mathbf 0, q_3\mathbf 1].
\]
Following  similar calculations to these for  $w_{2}$, we can prove that
\[
\frac{\partial}{\partial t} w_{3}-d_{3}k_{3}*w_{3}+d_{3}w_{3}-f_{3}(W)\leqslant 0.
\]
We also have that
\[
w_{3}(t,x)\geqslant M_3e^{-\alpha (t-1-\tau)}p(x)\quad\text{for}~t\geqslant 1+2\tau,~x\in\mathbb R.
\]
which provide the key inequality for the   construction of   $w_j$ with $j>3$ when $f_{j,3}>0$.

To construct  $w_j$ for $j\in\{4,5,\ldots,m\}$, when $f_{j,1}>0$, we apply the  construction method for $w_2$,
and when $f_{j,i}>0$ for some $i\in\{2,\ldots,j-1\}$,  we use the construction method for $w_3$ in the case $f_{3,2}>0$.
Then we can define every component of $W(t,x)$ satisfying
\[
\frac{\partial}{\partial t}W-DK*W+DW-F(W)\leqslant \mathbf 0\quad~\text{for}~t\geqslant1,~x\in\mathbb R,\\
\]
and
\[
w_{i}(t,x)\geqslant M_{i}e^{-\alpha [t-1-(i-2)\tau]}p(x)\quad\text{for}~t\geqslant1+(i-1)\tau,~x\in\mathbb R,~i=2,\ldots,m.
\]
We obtain  two constants
\[
T=1+(m-1)\tau\quad \text{and}~~ M_0=\left\{M_1 e^{-\alpha (T-1)},\min\limits_{i\in\{2,\ldots,m\}}\left\{M_i e^{-\alpha [T-1-(i-2)\tau]}\right\}\right\}
\]
such that
\begin{equation}\label{4.5}
w_{i}(T,x)\geqslant M_0 p(x),~x\in\mathbb R~\quad~\text{for any}~ i=1,\ldots,m.
\end{equation}
The definition of $W$  also shows that $w_{i}(1,x)=0$ for any $i=2,\ldots,m$.  We get from \eqref{4.3} that
\[
U(1,x)\geqslant W(1,x),~~~~x\in\mathbb R.
\]
It follows from  Lemma \ref{le-cp}   that
\[
U(t,x)\geqslant W(t,x)\quad\text{for}~t\geqslant1,~x\in\mathbb R.
\]
Then we have   $U(T,x)\geqslant W(T,x)$ for $x\in\mathbb R$, which implies \eqref{4.1} by \eqref{4.5}. It  completes the proof  of Theorem \ref{th1.3}.
\end{proof}

\section*{Acknowledgments}

The research of R. Hou was partially supported by China Postdoctoral Science Foundation funded project (BX20200011).
The research of Z.A. Wang was partially supported by an internal grant UAH0 (project id: P0031504) from the Hong Kong Polytechnic University.  The research of W.-B. Xu was partially supported by China Postdoctoral Science Foundation funded project (2020T130679) and by the CAS AMSS-POLYU Joint Laboratory of Applied Mathematics postdoctoral fellowship scheme.
The research of Z. Zhang  was partially supported by National Natural Science Foundation of China (11771428, 12026217, 12031015).

\end{document}